\documentclass[reqno]{amsart}
\usepackage[utf8]{inputenc}
\usepackage[T1]{fontenc}
\usepackage{lmodern}
\usepackage{hyperref}
\usepackage{amsmath}
\usepackage{amsfonts}
\usepackage{amssymb}
\usepackage{cancel}
\usepackage{color}
\usepackage{lipsum}

\numberwithin{equation}{section}
\newtheorem{theorem}{Theorem}[section]
\newtheorem{proposition}[theorem]{Proposition}
\newtheorem{remark}[theorem]{Remark}
\newtheorem{definition}{Definition}[section]
\newtheorem{claim}{Claim}[section]
\newtheorem{assumption}{Assumption}
\newtheorem{lemma}{Lemma}[section]

\newcommand{\<}{\langle}
\renewcommand{\>}{\rangle}
\renewcommand{\H}{\mathbb{H}}
\begin{document} 
\title{ Generalized eigenproblem and nonlinear elliptic
system with nonlinear boundary conditions}






\begin{abstract}
We will study solvability of nonlinear second-order elliptic system of partial
differential equations with nonlinear boundary conditions. We study the generalized
Steklov–Robin eigensystem (with possibly matrices weights) in which the spectral
parameter is both in the system and on the boundary. We prove the
existence of solutions for nonlinear system when both nonlinearities in the
differential system and on the boundary interact, in some sense, with the
generalized spectrum of Steklov-Robin. The method of proof makes use of Leray-Schauder degree.
\end{abstract}
\section{Notations definitions}
To put our results into the context, we have collected in this shore section some relevant notations and definitions for our purposes. 
Throughout this work, $H^{1}_{0},$ $H^{1}(\Omega)$ denotes the usual real Sobolev space of functions on $\Omega$. 
\begin{itemize}
 \item $k\in\mathbb{N}$
 \item $H(\Omega)=H^{1}(\Omega)\times H^{1}(\Omega)\times\cdots\times H^{1}(\Omega)\times H^{1}(\Omega)=[H^{1}(\Omega)]^{k}$
 \item $H_{0}(\Omega)=H^{1}_{0}(\Omega)\times H^{1}_{0}(\Omega)\times\cdots\times H^{1}_{0}(\Omega)\times H^{1}_{0}(\Omega)=[H^{1}_{0}(\Omega)]^{k}$
 \item $L^{p}_{k}(\Omega)=L^{p}(\Omega)\times L^{p}(\Omega)\times\cdots\times L^{p}(\Omega)\times L^{p}(\Omega)=[L^{p}(\Omega)]^{k}$
 \item  $L^{p}_{k}(\partial\Omega)=L^{p}(\partial\Omega)\times L^{p}(\partial\Omega)\times\cdots\times L^{p}(\partial\Omega)\times L^{p}(\partial\Omega)=[L^{p}(\partial\Omega)]^{k}$
\item $[W^{2}_{p}(\Omega)]^{k}=W^{2}_{p}(\Omega)\times W^{2}_{p}(\Omega)\times\cdots\times W^{2}_{p}(\Omega)\times W^{2}_{p}(\Omega)$
 \end{itemize}
\begin{definition}{\bf{Cooperative-plus}} 
Cooperative-plus matrix means that all the entries of the matrix are non-negative
\end{definition}

\maketitle
\section{Introduction}
This paper is concerned existence of (weak) solutions to the following nonlinear elliptic system 
\begin{equation}\label{Es1}
\begin{gathered}
 -\Delta U + A(x)U = f(x,U)\quad\text{in } \Omega,\\
\frac{\partial U}{\partial \nu}+\Sigma(x)U=g(x,U) \quad\text{on }
\partial\Omega,
\end{gathered}
\end{equation}

where $\Omega\subset\mathbb{R}^N$ , $N\geq 2$ is a bounded domain  with boundary
$\partial \Omega$ of class $C^{0,1}$,

$U=[u_1,\dots,u_k]^{T}\in H(\Omega)$
and the matrix 
$$A(x)=\left[
\begin{array}{cccc}
 a_{11}(x)&a_{12}(x)&\cdots&a_{1k}(x)\\ 
a_{21}(x)&a_{22}(x)&\cdots&a_{2k}(x)\\ 
\vdots & \vdots  &\ddots &\vdots  \\ 
a_{k1}(x)&a_{k2}(x)&\cdots&a_{kk}(x)
\end{array}\right]
$$\\
Verifies the following conditions:
\begin{enumerate}

\item[$(A1$)] The functions $a_{ij}:\Omega\to\mathbb{R},$ $a_{ij}(x)\ge 0,~~ \forall~ i,j=1,\cdots,k, ~~x\in\Omega$
with strict inequality on a set of positive measure of $\Omega$.
 \item[$(A2)$] $A(x)$ is positive semidefinite matrix on $\mathbb{R}^{{k}\times{k}},$ almost everywhere $x\in\Omega,$ and 
 $A(x)$ is positive definite on a set of positive measure of $\Omega$ with $a_{ij}\in L^{p}(\Omega)~~ \forall ~i,j=1,\cdots,k,$ 
for $p>\frac{N}{2}$ when $N\geq 3$, and $p>1$ when $N=2$
 
 \end{enumerate}
 The function $F:\bar{\Omega}\times\mathbb{R}^{k}\to\mathbb{R}^{k},$ 
 $$f(x,U)=f(x,u_1,\cdots,u_{k})=\begin{bmatrix}
       f_{1}(x,U)\\[0.3em]
       \vdots\\
       f_{k}(x,U)
     \end{bmatrix}$$
 Satisfies  the following conditions:\\
 \begin{enumerate}
  \item[$(F1)$] $F$ is Carath\'{e}dory function 
  \item[$(F2)$] There exist constants $b1>0,$ such that
  $$|f(x,U)|\leq b_1(1+||U||_{L^{p}_{k}(\Omega)}^{q_{1}-1}),~~{\rm with}~~2\leq q_{1} \leq \frac{2N}{N-2}, $$
  where $||U||=\displaystyle\sum_{i=1}^{k}||u_{i}||_{L^{p}(\Omega)}$
 \end{enumerate}
$\partial/\partial\nu :=\nu\cdot\nabla$
is the outward (unit) normal derivative on $\partial\Omega$,
The matrix $\Sigma$ is 
$$\Sigma(x)=\left[
\begin{array}{cccc}
 \sigma_{11}(x)&\sigma_{12}(x)&\cdots&\sigma_{1k}(x)\\ 
\sigma_{21}(x)&\sigma_{22}(x)&\cdots&\sigma_{2k}(x)\\ 
\vdots & \vdots  &\ddots &\vdots  \\ 
\sigma_{k1}(x)&\sigma_{k2}(x)&\cdots&\sigma_{kk}(x)
\end{array}\right]
$$
Verifies the following conditions:\\

\begin{enumerate}

\item[$(S1)$] The functions $\sigma_{ij}:\partial\Omega\to\mathbb{R},$ $\sigma_{ij}(x)\ge 0,~~ \forall ~i,j=1,\cdots,k ~~x\in\partial\Omega$
\item[$(S2)$] $\Sigma(x)$ is positive semidefinite matrix on $\mathbb{R}^{{k}\times{k}},$ almost everywhere $x\in\partial\Omega,$
and  $\Sigma(x)$ is positive definite on a set of positive measure of $\partial\Omega$
  with $\sigma_{ij}\in L^{q}(\partial\Omega)~~ \forall~ i,j=1,\cdots,k,$ 
for $q\geq N-1$ when $N\geq 3$, and $q>1$ when $N=2$

\begin{remark}
 \item Conditions $(A2),(S2)$ are equivalent to 
 $$\int_{\Omega}\<A(x)U,U\>\,dx+\int_{\partial\Omega}\<\Sigma(x)U,U\>\,dx > 0$$
\item We can put condition $(S2)$ in $A$ and condition $(A2)$ in $\Sigma$ i.e.;  interchange conditions in $A$ and $\Sigma$ 
 \end{remark}

\end{enumerate}
$G:\partial\Omega\times\mathbb{R}^{k}\to\mathbb{R}^{k},$ 
 $$g(x,U)=g(x,u_1,\cdots,u_{k})=\begin{bmatrix}
       g_{1}(x,U)\\[0.3em]
       \vdots\\
       g_{k}(x,U)
     \end{bmatrix}$$
 satisfies the following conditions
 \begin{enumerate}
  \item[$(G1)$] $G$ is Carath\'{e}dory function 
  \item[$(G2)$] There exist constants $a_1>0$ such that
  $$|g(x,U)|\leq a_1(1+||U||^{q_{2}-1}_{L^{p}_{k}(\partial\Omega)})~~{\rm with}~~2\leq q_{2} \leq \frac{2N}{N-2} ,$$
  where $$||U||{L^{p}_{k}(\partial\Omega)}=\sum_{i=1}^{k}||u_i||_{L^{p}(\partial\Omega)}.$$
  \item[$(G3)$] And for every constant
$r>0$ there is a constant $K=K(r)$ such that 
$$|g(x,U)-g(x,V)|\leq K(|x-y|+|U-V|)$$ for all $x,y\in \partial\Omega$ and all $U,V\in H(\Omega)$ with $|U|\leq r,$ and $|V|\leq r$
  \end{enumerate}
 Where the nonlinear reaction-function $f(x,U)$ and nonlinearity on the boundary $g(x,U)$ interact, in some sense, with the generalized Steklov-Robin
 spectrum of the following linear system problem (with possibly  matrices $(M,P)$-weights)
  
\begin{equation}\label{Es2}
\begin{gathered}
 -\Delta U+ A(x)U =\mu M(x)U \quad\text{in } \Omega,\\
\frac{\partial U}{\partial \nu}+\Sigma(x)U=\mu P(x)U \quad\text{on }
\partial\Omega,
\end{gathered}
\end{equation}
Where 
$$M(x)=\left[
\begin{array}{cccc}
 m_{11}(x)&m_{12}(x)&\cdots&m_{1k}(x)\\ 
m_{21}(x)&m_{22}(x)&\cdots&m_{2k}(x)\\ 
\vdots & \vdots  &\ddots &\vdots  \\ 
m_{k1}(x)&m_{k2}(x)&\cdots&m_{kk}(x)
\end{array}\right]
$$
and 
$$P(x)=\left[
\begin{array}{cccc}
 \rho_{11}(x)&\rho_{12}(x)&\cdots&\rho_{1k}(x)\\ 
\rho_{21}(x)&\rho_{22}(x)&\cdots&\rho_{2k}(x)\\ 
\vdots & \vdots  &\ddots &\vdots  \\ 
\rho_{k1}(x)&\rho_{k2}(x)&\cdots&\rho_{kk}(x)
\end{array}\right]
$$
Note that the eigensystem (\ref{Es2}) includes as special cases the weighted Steklov eigenproblem for a class of elliptic system when ( 
$\Sigma(x)=M(x)\equiv 0$, $A(x)\in\mathbb{R}^{2\times{2}}$ and  $P(x)=I$  where $I$ 
is identity matrix that was considered in \cite{GMR13}. For scalar  case 
generalized Steklov-Robin  spectrum of the scalar case that was considered in \cite{Mav12}
as well as weighted Robin-Neumann eigenproblem when ($P(x)\equiv 0$ and $M(x)=I$) that was considered in \cite{Auc12} and \cite{GMR13}. 
Under conditions of $A(x),\Sigma(x)$ together with the hypothesis on $\Omega$, we have 
\begin{itemize}

 \item The embedding of $H^{1}(\Omega)$ into $L^{p}(\Omega)$ is continuous for $1\leq p\leq p(N)$ and compact for 
 $1\leq p\leq p(N)$ where $p(N)=\frac{2N}{N-2}$ if $N\geq 3$ and $p(N)=\infty$ if $N=2$ (see \cite{EG} for more details)
   \item $\<x,y\>=\sum_{i=1}^{k}x_iy_i~~\forall~ x,y\in H(\Omega)$ then 
   
  \end{itemize}

The main difficulty is to lead with trace operator. In the scalar we would like to cite papers by 
\cite{GMR13}, \cite{Mav12}, \cite{MN}, \cite{MK2} and \cite{MK1}, while for the system case, to best of our knowledge, we not able to find any reference 
for (\ref{Es1}), which we describe herein for the first time.

\section{Generalized Steklov-Robin eigensystem}
In this section, we will study the generalized spectrum that will  be used for the comparison with the nonlinearities in the system (\ref{Es1}).
This spectrum includes the Steklov, Neumann and Robin spectra, We therefore generalize the results in \cite{Auc12}, and \cite{GMR13}\\.
Consider the elliptic system 
\begin{equation}\label{Es3}
\begin{gathered}
 -\Delta U+ A(x)U =\mu M(x)U \quad\text{in } \Omega,\\
\frac{\partial U}{\partial \nu}+\Sigma(x)U=\mu P(x)U \quad\text{on }
\partial\Omega,
\end{gathered}
\end{equation}
Where 
$$M(x)=\left[
\begin{array}{cccc}
 m_{11}(x)&m_{12}(x)&\cdots&m_{1k}(x)\\ 
m_{21}(x)&m_{22}(x)&\cdots&m_{2k}(x)\\ 
\vdots & \vdots  &\ddots &\vdots  \\ 
m_{k1}(x)&m_{k2}(x)&\cdots&m_{kk}(x)
\end{array}\right]
$$
and 
$$P(x)=\left[
\begin{array}{cccc}
 \rho_{11}(x)&\rho_{12}(x)&\cdots&\rho_{1k}(x)\\ 
\rho_{21}(x)&\rho_{22}(x)&\cdots&\rho_{2k}(x)\\ 
\vdots & \vdots  &\ddots &\vdots  \\ 
\rho_{k1}(x)&\rho_{k2}(x)&\cdots&\rho_{kk}(x)
\end{array}\right]
$$
\begin{enumerate}
 \item[(M1)] Where $M(x)$ is positive definite matrix on $\mathbb{R}^{{k}\times{k}}$, almost everywhere $x\in\Omega,$ 
The functions $m_{ij}:\Omega\to\mathbb{R},$ $m_{ij}(x)\ge 0,~~ \forall~ i,j=1,\cdots,k ~~x\in\Omega,$
with strict inequality on a set of positive measure of $\Omega$, $m_{ij}\in L^{p}(\Omega)~~ \forall~ i,j=1,\cdots,k,$ 
 for $p\geq \frac{N}{2}$ when $N\geq 3$, and $p>1$ when $N=2.$ 
\item[(P1)] Where $P(x)$ is positive semidefinite matrix on $\mathbb{R}^{{k}\times{k}},$ almost everywhere $x\in\partial\Omega,$ and positive define
 on a set of positive measure of $\partial\Omega$ with $\rho_{ij}\in L^{q}(\partial\Omega)~~ \forall i,j=1,\cdots,k,$ 
for $q\geq N-1$ when $N\geq 3$, and $q>1$ when $N=2.$
\item[(MP)] And $(m_{ij},\rho_{ij})>0$; that is  $m_{ij}(x)\geq 0$ a.e in $\Omega$, $\rho_{ij}(x)\geq 0$ a.e on $\partial\Omega$ for all $i,j=1,\cdots,k$
such that $$\int_{\Omega}m_{ij}(x)dx+\int_{\partial\Omega}\rho_{ij}(x)dx >0 ~\forall~i,j=1,\cdots,k,$$
 
\end{enumerate}
Assume that $A(x),~\Sigma(x),~M(x),~P(x)$ are Verifies the following assumption :\\
\begin{assumption}\label{ASMP}
 \begin{enumerate}
$A(x)$ is positive definite on a set of positive measure of $\Omega$ with $A(x)\in L^{p}(\Omega)$ 
for $p>\frac{N}{2}$ when $N\geq 3$, and $p>1$ when $N=2.$ 
\item [OR] 
  $\Sigma(x)$ is positive definite on a set of positive measure of $\partial\Omega$
  with $\Sigma(x)\in L^{q}(\partial\Omega)$
for $q\geq N-1$ when $N\geq 3$, and $q>1$ when $N=2$
\item[AND] 
$M(x)$ is positive definite on a set of positive measure of $\Omega$ with $M(x)\in L^{p}(\Omega)$ 
for $p>\frac{N}{2}$ when $N\geq 3$, and $p>1$ when $N=2.$ 
\item[OR]
$P(x)$ is positive definite on a set of positive measure of $\partial\Omega$
  with $P(x)\in L^{q}(\partial\Omega)$
for $q\geq N-1$ when $N\geq 3$, and $q>1$ when $N=2.$
\end{enumerate}
\end{assumption}
\begin{remark}\label{rm1}
 Since $A(x),~\Sigma(x),~M(x),~P(x)$ are satisfied $(A2),~ (S2),~ (M1), ~(P1)$ respectively and they are Cooperative-plus matrices,
 then we can write them in following from 
 (i.e.;  eigen-decomposition of a positive semi-definite matrix or diagonalization)
 $$A(x)=Q_{A}^{T}(x)D_{A}(x)Q_{A}(x).$$
$$\Sigma(x)=Q_{\Sigma}^{T}(x)D_{\Sigma}(x)Q_{\Sigma}(x).$$
$$M(x)=Q_{M}^{T}(x)D_{M}(x)Q_{M}(x).$$
$$P(x)=Q_{P}^{T}(x)D_{P}(x)Q_{P}(x).$$
Where $Q_{J}(x)^{T}Q_{J}(x)=I$ ($Q_{J}^{T}(x)=Q_{J}^{-1}(x)$ i.e.; are orthogonal matrix ) are the normalized eigenvectors,  $I$ is identity matrix, $D_{J}(x)$ is diagonal
matrix in the diagonal of $D_{J}(x)$ the eigenvalues of $J(x)$ (i.e.; $D(x)_{J}=diag\left(\lambda_{1}^{J}(x),\cdots,\lambda_{k}^{J}(x)\right)$)
and $J(x)=\{A(x),~\Sigma(x),~M(x),~P(x)\}$
 \end{remark}

\begin{remark}
 The weight matrices $M(x)$ and $P(x)$ may vanish on subset of positive measure. 
\end{remark}
\begin{definition}
 The generalized Steklov-Robin eigensystem is to find a pair $(\mu,\varphi)\in\mathbb{R}\times H(\Omega)$ with 
 $\varphi\not\equiv 0$ such that 
 \begin{equation}\label{Es4}
\begin{gathered}
 \int_{\Omega}\triangledown\varphi.\triangledown U\,dx+\int_{\Omega}\<A(x)\varphi,U\>\,dx+\int_{\partial\Omega}\<\Sigma(x)\varphi,U\>\,dx= \\
 \mu\left[\int_{\Omega}\<M(x)\varphi,U\>\,dx+\int_{\partial\Omega}\<P(x)\varphi,U\>\,dx\right] ~~\forall~~U\in H(\Omega)
\end{gathered}
\end{equation}
\end{definition}
\begin{remark}
 Let $U=\varphi$ in (\ref{Es4}), if there is such an eigenpair, then 
$\mu>0$ and 
$$\int_{\Omega}\<M(x)\varphi,\varphi\>\,dx+\int_{\partial\Omega}\<P(x)\varphi,\varphi\>\,dx >0$$
$$\left(i.e; \sum_{j=1}^{k}\sum_{i=1}^{k}\left(\int_{\Omega}m_{ij}(x)\varphi_{j}\varphi_{i}\,dx+\int_{\partial\Omega}\rho_{ij}(x)\varphi_{j}\varphi_{i}\,dx\right)>0\right)$$ and 
\end{remark}
\begin{remark}
If $\int_{\Omega}\<M(x)\varphi,\varphi\>\,dx+\int_{\partial\Omega}\<P(x)\varphi,\varphi\>\,dx =0.$
Then 
$$\int_{\Omega}|\triangledown\varphi|^{2}\,dx+\int_{\Omega}\<A(x)\varphi,\varphi\>\,dx+\int_{\partial\Omega}\<\Sigma(x)\varphi,\varphi\>\,dx=0$$
We have that 
$\int_{\Omega}|\triangledown\varphi|^{2}\,dx=0$ this implies that $\varphi={\rm~ constant}$
and 

$\int_{\Omega}\<A(x)\varphi,\varphi\>\,dx=0$ this implies that $A(x)=0,~a.e.$  in $\Omega$ and 
$\int_{\partial\Omega}\<\Sigma(x)\varphi,\varphi\>\,dx =0,$ then  $\Sigma(x)=0~a.e$ on $\partial\Omega.$
So we have that, $\varphi$ would be a constant vector function; which would contradict the assumptions (assumption \ref{ASMP})  imposed on 
$A(x)$ and $\Sigma(x)$
\end{remark}
\begin{remark}
If $A(x)\equiv 0$ and $\Sigma\equiv 0$ then $\mu=0$ is an eigenvalue of the system (\ref{Es3}) with eigenfunction 
$\varphi={\rm constant~ vector ~function}$ on $\bar{\Omega},$
\end{remark}
It is therefore appropriate to consider the closed linear subspace of $H(\Omega)$ under assumption \ref{ASMP} defined by 
$$\H_{(M,P)}(\Omega):=\{U\in H(\Omega):\int_{\Omega}\<M(x)U,U\>\,dx+\int_{\partial\Omega}\<P(x)U,U\>\,dx=0\}.$$
Now all the  eigenfunctions associated with (\ref{Es4}) belongs  to the $(A,\Sigma)-$orthogonal complement $H_{(M,P)}(\Omega):=[\H_{(M,P)}(\Omega)]^{\perp}$ 
$$i.e.;H_{(M,P)}(\Omega)=[\H_{(M,P)}(\Omega)]^{\perp}=$$
$$\{U\in H(\Omega):\int_{\Omega}\<M(x)U,U\>\,dx+\int_{\partial\Omega}\<P(x)U,U\>\,dx>0~{\rm and} ~~\<U,V\>_{(A,\Sigma)}=0, \forall~V\in \H_{(M,P)}(\Omega)\}$$
of this subspace in $H(\Omega)$
\begin{definition}
 $\Omega(M):=\{x\in\Omega:M(x) >0\}.$ \\ 
  $\partial\Omega(P):=\{x\in\partial\Omega:P(x) >0\}.$ 
\end{definition}
We will show that indeed the  $H_{(M,P)}(\Omega)$ is subspace of $H(\Omega).$
Let $U,~ V\in H_{(M.P)}(\Omega)$ and $\alpha\in\mathbb{R}$ we show that $\alpha U\in H_{(M.P)}(\Omega)$ and $U+V\in H_{(M.P)}(\Omega) $
$$\left(\int_{\Omega}\< M(x)(\alpha U),\alpha U\>\,dx+\int_{\partial\Omega}\<P(x)(\alpha U),\alpha  U\>\,dx\right)=$$
$$\alpha^{2}\left(\int_{\Omega}\<M(x)U,U\>\,dx+\int_{\partial\Omega}\<P(x)U,U\>\,dx\right)=^{U\in H_{(M.P)}}0$$
Therefore $\alpha U \in H_{(M.P)}(\Omega)$. Now we show that $U+V\in H_{(M.P)}(\Omega) $
$$\int_{\Omega}\< M(x)(U+V),(U+V)\>\,dx+\int_{\partial\Omega}\<P(x)(U+V),(U+V)\>\,dx=$$
$$\int_{\Omega}\<M(x)U,U\>\,dx+\int_{\partial\Omega}\<P(x)U,U\>\,dx+\int_{\Omega}\<M(x)V,V\>\,dx+\int_{\partial\Omega}\<P(x)V,V\>\,dx$$
$$+2\int_{\Omega}\< M(x)U,V\>\,dx+2\int_{\partial\Omega}\<P(x)U,V)\>\,dx$$
{\bf{Case 1}}\\
If $x\in  \partial\Omega(P),$  this implies that   $U=0$ on $\partial\Omega(P)$
\begin{itemize}
\item $M(x)=0$  on $\Omega$ (Steklov problem)
\item $M(x)>0$ on set of positive measure of $\Omega$ this implies that $U=0$ on $\Omega$
Now we use  and \ref{ASMP}, \ref{rm1} we have that 
$$0=\int_{\Omega}\<M(x)U,U\>\,dx=\int_{\Omega}\<Q^{T}_{M}(x)D_{M}(x)Q_{M}(x)U,U\>\,dx=
\int_{\Omega}\<D_{M}(x)Q_{M}(x)U,Q_{M}(x)U\>\,dx
$$ 
Since $D_{M}(x)>0$ this implies that $Q_{M}(x)U=0$ since $Q_{M}(x)$ is invertible this implies that 
$U=0$ on $\Omega(M)$
\end{itemize}
{\bf{Case 2}}\\
If $x\in \Omega(M),$  this implies that   $U=0$ on $\Omega$
\begin{itemize}
\item $P(x)=0$  on $\partial\Omega$ (usual spectrum problem ) 
\item $P(x)>0$ on set of positive measure of $\partial\Omega$ this implies that $U=0$ on $\partial\Omega$
Now we use  and \ref{ASMP}, \ref{rm1} we have that 
$$0=\int_{\partial\Omega}\<P(x)U,U\>\,dx=\int_{\partial\Omega}\<Q^{T}_{P}(x)D_{P}(x)Q_{P}(x)U,U\>\,dx=
\int_{\partial\Omega}\<D_{P}(x)Q_{P}(x)U,Q_{P}(x)U\>\,dx
$$ 
Since $D_{P}(x)>0$ this implies that $Q_{P}(x)U=0$ since $Q_{P}(x)$ is invertible this implies that 
$U=0$ on $\Omega(P)$
\end{itemize}
Therefor $U+V\in \H_{(M,P)}(\Omega), $ so we have that $\H_{(M,P)}(\Omega)$ subspace of $H(\Omega).$ 
So that 
$$\H_{(M,P)}(\Omega):=\{U\in H(\Omega):U=0~a.e.;~{\rm in}~\Omega(M)~{\rm and}~\Gamma U=0~a.e.;{\rm on}~~ \partial\Omega(P)\}$$
\begin{remark}.\\
\begin{enumerate}
 \item If $M(x)\equiv 0$ in $\Omega$ and $x\in\partial\Omega(P)$, then the subspace $\H_{(M,P)}(\Omega)=H_{0}(\Omega)$
\item If $P(x)\equiv 0$ in $\partial\Omega$ and $x\in\Omega(M)$, then the subspace $\H_{(M,P)}(\Omega)=\{0\}$
\end{enumerate}
\end{remark}
Thus, one can split the Hilbert space $H(\Omega)$ as a driect $(A,\Sigma)-$orthogonal sum in the following way 
$$H(\Omega)=\H_{(M,P)}(\Omega)\oplus_{(A,\Sigma)}[\H_{(M,P)}(\Omega)]^{\perp}$$
\begin{remark}
 If $(\mu,\varphi)\in\mathbb{R}\times H(\Omega)$ is an eigenpair of \ref{Es4}, then it follows from the definition of $H_{0}(\Omega)$
that 
$$<\varphi,V>_{(A,\Sigma)}=
\int_{\Omega}[\triangledown\varphi.\triangledown V+\<A(x)\varphi,V\>]\,dx
+\int_{\partial\Omega}\<\Sigma(x)\varphi,V\>\,dx=0,~~~$$
$\forall ~V\in \H_{(M,P)}(\Omega)$ that is, $\varphi\in [\H_{(M,P)}(\Omega)]^{\bot} $
\end{remark}
\begin{itemize}
  \item We shall make use in what follows the real Lebesgue space $L^{q}_{k}(\partial\Omega)$ for $1\leq q\leq\infty$, and 
of the continuity and compactness of the trace operator. 
$$\Gamma :H(\Omega)\to L^{q}_{k}(\partial\Omega)~~for~~1\leq q<\frac{2(n-1)}{n-2}$$
 is well-defined it is a Lebesgue integrable function  with respect to Hausdorff $N-1$ dimensional measure, 
 sometime we will just use $U$ in place of $\Gamma U$ when considering the trace of function on $\partial\Omega$
Throughout  this work we denote the $L^{2}_{N}(\partial\Omega)-$ inner product by 
$$\<U,V\>_{\partial}:=\int_{\partial\Omega}U.V\,dx$$ and  the associated norm by 
$$||U||^{2}_{\partial}:=\int_{\partial\Omega}U.U~\forall U,V\in H(\Omega)$$
(see \cite{KJF77}, \cite{Nec67} and the references therein for more details )
  \item The trace mapping $\Gamma:H(\Omega)\to L^{2}_{N}(\partial\Omega)$ is compact (see \cite{Gri})
  \item 
  \begin{equation}\label{IMP}
\begin{gathered}
\<U,V\>_{(M,P)}=\int_{\Omega}\<M(x)U,V\>\,dx+\int_{\partial\Omega}\<P(x) U,V\>\,dx
\end{gathered}
\end{equation}
    defines an inner product for $H(\Omega)$, with associated norm 
  \begin{equation}\label{NMP}
\begin{gathered}
  ||U||_{(M,P)}^{2}:=\int_{\Omega}\<M(x)U,U\>\,dx+\int_{\partial\Omega}\<P(x) U,U\>\,dx
\end{gathered}
\end{equation}
\end{itemize}

 \begin{theorem}\label{ths}
Assume that $A2,S2,C1-C3$ as above. Then we have the following. 
\begin{description}
 \item[i] The eigensystem \eqref{Es3} has a sequence of real eigenvalues 
$$0<\mu_{1}\leq\mu_2\leq\mu_3\leq\cdots\leq\mu_j\leq\cdots\to\infty~as~j\to\infty$$ 
each eigenvalue has a finite-dimensional eigenspace. 
\item[ii] The eigenfunctions $\varphi_j$ corresponding to the eigenvalues $\mu_j$ from an $(A,\Sigma)-$orthogonal and $(M,P)-$orthonormal family in 
$[\mathbb{H}_{M,P}(\Omega)]^{\bot}$ ( a closed subspace of $H(\Omega)$
\item[iii] The normalized eigenfunctions provide a complete $(A,\Sigma)-$orthonormal basis of $[\mathbb{H}_{M,P}(\Omega)]^{\bot}$. 
Moreover, each function in $U\in [\mathbb{H}_{M,P}(\Omega)]^{\bot}$
has a unique representation of the from 
\begin{equation}\label{4}
\begin{gathered}
U=\displaystyle\sum^{\infty}_{j=1}c_j\varphi_j~with~c_j:=\frac{1}{\mu_j}\<U,\varphi_j\>_{(A,\Sigma)}=<U,\varphi_j>_{(M.P)} \\
||U||^{2}_{(A,\Sigma)}=\displaystyle\sum^{\infty}_{j=1}\mu_j|c_j|^2
\end{gathered}
\end{equation}
In addition, $$||U||^{2}_{(M,P)}=\displaystyle\sum^{\infty}_{j=1}|c_j|^2$$
\end{description}
\end{theorem}

The following proposition shows the principality of the first eigenvalue $\mu_{1}.$
\begin{proposition}\label{Ep1}
The first eigenvalue $\mu_{1}$ is simple if and only if the associated eigenfunction $\varphi_{1}$ does not changes sign (i.e.; 
$\varphi_{1}$ is strictly positive or strictly negative in $\Omega$ .
\end{proposition}

\begin{remark}
Note that if we have smooth data and $\partial\Omega$ in proposition \ref{Ep1}, then the eigenfunction $\varphi_{1}(x)$ on $\partial\Omega$ as well, 
by the boundary point lemma (see for example \cite{Eva99}). or Hopf's Boundary Point Lemma.
\end{remark}

\begin{definition}{\bf{Reflexive Space}}
 Let $E$ be a Banach space and let $J:E\to E^{**}$ ( Where The Bidual $E^{**}$ Orthogonality )be the canonical injection from $E$ into $E^{**}.$ The space $E$
 is said to be reflexive if $J$ is surjective, i.e; $J(E)=E^{**}.$ (See \cite{Bz})  
\end{definition}
\begin{theorem}\label{rex}
 Assume that $E$ is a reflexive Banach space and let $\{x_{n}\}_{n=1}^{\infty}\subset E$ be bounded sequences in $E.$ 
 Then there exists a subsequence $\{x_{n_{l}}\}_{l=1}^{\infty}\subset E$ that converges in the weak topology  $\sigma (E,E^{*})$ ( for proof see \cite{Bz})
\end{theorem}

   \section{Main Results}   
   
We will consider special case from equation \ref{Es1} when the $f(x,U)=0$ and  when $M(x)=0$ in equation \ref{Es3} 
then we have that

\begin{equation}\label{Es01}
\begin{gathered}
 -\Delta U + A(x)U = 0\quad\text{in } \Omega,\\
\frac{\partial U}{\partial \nu}+\Sigma(x)U=g(x,U) \quad\text{on }
\partial\Omega,
\end{gathered}
\end{equation}

\begin{equation}\label{Es02}
\begin{gathered}
 -\Delta U+ A(x)U =0\quad\text{in } \Omega,\\
\frac{\partial U}{\partial \nu}+\Sigma(x)U=\mu P(x)U \quad\text{on }
\partial\Omega,
\end{gathered}
\end{equation}
We obtain solution for the problem (\ref{Es01}) which the nonlinearities at the boundary are compared with higher eigenvalues 
of equation(\ref{Es02}), we obtain form Theorem\ref{th1} problem (\ref{Es02}) has sequence of real eigenvalues 
$$0<\mu_{1}\leq\mu_2\leq\mu_3\leq\cdots\leq\mu_j\leq\cdots\to\infty~as~j\to\infty$$ 
each eigenvalue has a finite-dimensional eigenspace. 
In this section we impose conditions on the asymptotic behavior of the 'slopes' of the boundary nonlinearities $g(x,U)$ i.e.; 
on $\frac{\<g(x,U),U\>}{|U|^{2}}$ as $|U|\to\infty.$ These conditions are of nonuniform type since the asymptotic ration 
$\frac{\<g(x,U),U\>}{|U|^{2}}$ need not be (uniformly) bounded away from consecutive Steklov-Robin eigenvalues. We mention 
that results below, the boundary nonlinearity $g(x,U)$ may be replaced by $g(x,U)+H(x)$ where $H\in [W^{1-1/p}_{P}(\partial\Omega)]^{k}
\subset [C(\partial\Omega)]^{k}.$ We approach which based in on topological degree
theory on suitable boundary-trace spaces, here in this work we generalization of \cite{MK2} from Scalar case to system case. 
If for the scalar case this is in contrast recent approaches where variational methods were used for problems with nonlinear 
conditions, but in the case of system 
the variational methods doesn't work because of the potentials doesn't exists as integrations of $f(x,U)$ and $g(x,U)$, 
since that potential function the hart 
of the variational methods therefore doesn't work. In case  $f(x,U)\neq 0$ the system will be very hard to study even in in the 
topological degree methods and comparison
with eigenvalues of Steklov-Robin that the reason we will study the case when $(Fx,U)=0$ and we will leave the case when $f(x,U)\neq 0$ 
as  the future work. 
\begin{remark}
 In the scalar when $f(x,U)\neq$ comparison with eigenvalues of Steklov-Robin used the variational method was done by Mavinga  see 
\cite{Mav12} of strong Solution for equation (\ref{Es1}) in scalar case . For Multiple solution (weak) for equation ( \ref{Es1}) in scalar case and $\Sigma(x)=0$  used  variational Methods
was done by Yao see \cite{Yao2014}.
\end{remark}
\begin{remark}
Let $U\neq 0$, and assume that $(P1)$ hold,  the eigenvalue of $P(x)$ are nonnegative Let the $\lambda_{p}(x)\geq 0$ be the eigenvalue of $P(x)$ from 
  Rayleigh quotient we have that $\min\lambda_{P}(x)\leq\frac{\<P(x)U,U\>}{\<U,U\>}$ and $\max\lambda_{P}(x)\geq\frac{\<P(x)U,U\>}{\<U,U\>}$
\end{remark}
\subsection{Main Theorem}
The main result of this chapter is given in the following existence theorem. 
\begin{theorem}\label{th0}{\bf{Nonuniform  non-resonance between consecutive Steklov-Robin eigenvalues}}
Assume that $(A1),~(A2),~(S1),~(S2)~, (G1)-(G3)$ and assumption $\ref{ASMP}$ are holds,  assume there exist constants $a,b>0$,  
and Let $\alpha(x):=a(\min\lambda_{P}(x))$ and ,$\beta(x)=b(\max\lambda_{P}(x))$  such that 
$$\mu_{j}\leq\alpha(x)\leq\liminf_{|U|\to\infty}\frac{\<g(x,U),U\>}{|U|^{2}}\leq\limsup_{|U|\to\infty}\frac{\<g(x,U),U\>}{|U|^{2}}\leq \beta(x)\leq \mu_{j+1}$$
 uniformly for a.e. $x\in \partial\Omega$ with 
 $$\int_{\partial\Omega}(\alpha(x)-\mu_{j})\<\varphi(x),\varphi(x)\>\,dx>0,~~~\forall ~~\varphi\in E_{j+1}\setminus\{0\}$$
and $$~~\int_{\partial\Omega}(\mu_{j+1}-\beta(x))\<\varphi(x),\varphi(x)\>\,dx>0,~~\forall~~ \varphi\in E_{j+1}\setminus\{0\}, 
  $$ 
  Where $i\in \mathbb{N}, $ $E_{i}$ denotes the Steklov-Robin nullspace (finite-dimensional) with the Steklov-Robin
  eigenvalue $\mu_{i}$. Then the nonlinear system (\ref{Es02}) has at least one (strong) solution $U\in [W^{2}_{p}(\Omega)]^{k}$
 \end{theorem}
 In contrast to some recent approaches in the literature for system which nonlinear boundary conditions , we first the system 
 in terms of nonlinear compact perturbations of the identity on appropriate boundary-trace as follows
\section{THE HOMOTOPY}
Let $\delta$ between and consecutive Steklov-Robin eigenvalues ($i.e.; \mu_{j}<\delta<\mu_{j+1})$ and $\delta$ isn't eigenvalue of Steklov-Robin system.
system \ref{Es01}  equivalently 
\begin{equation}\label{e01}
\begin{gathered}
 -\Delta U + A(x)U = 0\quad\text{in } \Omega,\\
\frac{\partial U}{\partial \nu}+\Sigma(x)U-\delta U=g(x,U)-\delta U \quad\text{on }
\partial\Omega,
\end{gathered}
\end{equation}

Let us associate to equation (\ref{e01}) the homotopy, with $\lambda\in[0,1]$ 

\begin{equation}\label{hs01}
\begin{gathered}
 -\Delta U + A(x)U = 0\quad\text{in } \Omega,\\
\frac{\partial U}{\partial \nu}+\Sigma(x)U-\delta P(x) U=\lambda[g(x,U)-\delta P(x)U] \quad\text{on }
\partial\Omega,
\end{gathered}
\end{equation}
or equivalent, 
\begin{equation}\label{hs02}
\begin{gathered}
 -\Delta U + A(x)U = 0\quad\text{in } \Omega,\\
\frac{\partial U}{\partial \nu}+\Sigma(x)U=(1- \lambda) \delta P(x)U+\lambda g(x,U) \quad\text{on }
\partial\Omega,
\end{gathered}
\end{equation}
Note that for $\lambda=0$ we have that we have a linear system which admits only the trivial solutions since 
$\delta$ is in the resolvent of the linear Steklov-Robin problem ( for scalar case see \cite{Mav12} and \cite{Auc12} and the references therein ). 
Whereas  for $\lambda=1$ we have system (\ref{Es01})
We define the linear (Steklov-Robin) boundary operator 
$$\mathbb{L}:Dom(\mathbb{L})\subset [W^{2}_{p}(\Omega)]^{k}\Subset [W^{1-1/p}_{p}(\partial\Omega)]^{k}\to [W^{1-1/p}_{p}(\partial\Omega)]^{k}$$ ('$\Subset$' 
compact embedding. Here, the compact `containment' $ [W^{2}_{p}(\Omega)]^{k}\Subset [W^{1-1/p}_{p}(\partial\Omega)]^{k}$ must be 
understood in the sense of the trace (i.e.; $[W^{2}_{p}(\Omega)]^{k}\hookrightarrow[W^{1-1/p}_{p}(\partial\Omega)]^{k}$ is a 
compact linear operator ( see \cite{Eva99} and  \cite{GT83} ) )  
by 
\begin{equation}\label{l}
\begin{gathered}
\mathbb{L}U:=\frac{\partial U}{\partial \nu}+\Sigma(x)U-\delta P(x) U
\end{gathered}
\end{equation}

where 
$$Dom(\mathbb{L}):=\{U\in [W_{p}^{2}]^{k}:-\Delta U + A(x)U = 0 ~a.e. ~{\rm in}~ \Omega\}$$
\subsection{Nemytsk\v{i}i Operator}
 $$\v{N}:[W^{1-1/p}_{p}(\partial\Omega)]^{k}\subset[C(\partial\Omega)]^{k}\to [W^{1-1/p}_{p}(\partial\Omega)]^{k}$$ by 
 \begin{equation}\label{n}
\begin{gathered}
  \v{N}U:=G(.,U)-\delta P(.)U
\end{gathered}
\end{equation}

 System \ref{Es01} is then equivalent to finding $U\in Dom(\mathbb{L})$ such that

 \begin{equation}\label{nl}
\begin{gathered}
 \mathbb{L}U=\v{N}U
\end{gathered}
\end{equation}
Whereas the homotopy system\ref{hs02} is equivalent to 

\begin{equation}\label{nlh}
\begin{gathered}
 \mathbb{L}U=\lambda\v{N}U \,\,\,\,\, U\in Dom(\mathbb{L})
\end{gathered}
\end{equation}
 \section{Properties of Operators $\mathbb{L}, ~$\v{N}$~,\mathbb{L}^{-1}$ }
\begin{description}
 \item[I] $Dom(\mathbb{L}):=\{U\in [W_{p}^{2}]^{k}:-\Delta U + A(x)U = 0 ~a.e. ~{\rm in}~ \Omega\}$ is a closed linear subspace of 
 $[W_{p}^{2}]^{k}$
 \begin{proof}
  Clearly $Dom(\mathbb{L})\subset [W_{p}^{2}]^{k}$ is not empty (if $U=0$ then $0\in Dom(\mathbb{L})$) , let $U,W\in Dom(\mathbb{L})$ then we have that 
  $-\Delta (U+W) + A(x)(U+W) =-\Delta U + A(x)U -\Delta W + A(x)W =0,$ therefore  $U+W\in Dom(\mathbb{L}),$ 
  now let $\alpha\in\mathbb{R}$ such that 
  $-\Delta \alpha U + \alpha A(x)U=\alpha(-\Delta U + A(x)U)=0$ so that $\alpha U \in Dom(\mathbb{L}).$ Therefore 
  $Dom(\mathbb{L})$ is linear subspace of $[W_{p}^{2}]^{k}.$
  For the closed let $\{U_{n}\}_{n=1}^{\infty}\in Dom(\mathbb{L}),$ and $U_{n}\to U $ in $[W_{p}^{2}]^{k}.$ then 
  $-\Delta U_{n}\to-\Delta U$ and $A(x)U_{n}\to A(x)U$ as $n\to\infty$ then $-\Delta U+A(x)U=0$ so that $U\in Dom(\mathbb{L})$ 
  Therefore $Dom(\mathbb{L})$ is closed linear subspace of $[W_{p}^{2}]^{k}.$  
   \end{proof}
   \item[II] $\ker{\mathbb{L}}=0$, onto, and continuous  
   \begin{proof} Clearly is continuous for the definition 
    Let $U\in\ker{\mathbb{L}}\subset Dom(\mathbb{L})$ this implies that $-\Delta U + A(x)U)=0$ and by the definition of $\mathbb{L}$ (\ref{l}) 
    we have that 
   \begin{equation}\label{l0}
\begin{gathered}
\mathbb{L}U=0=\frac{\partial U}{\partial \nu}+\Sigma(x)U-\delta P(x) U \, x\in \partial\Omega
\end{gathered}
\end{equation}
 Then you have that 
\begin{equation}\label{l00}
\begin{gathered}
-\Delta U + A(x)U = 0 ~a.e. ~{\rm in}~ \Omega \\
\frac{\partial U}{\partial \nu}+\Sigma(x)U=\delta P(x) U \,x\in \partial\Omega
\end{gathered}
\end{equation}
system \ref{l00} has trivial solution. Therefore the $\ker{\mathbb{L}}=0$ i.e.; $\mathbb{L}$ one-to-one, Since the domain of $\mathbb{L}$ 
closed linear and  $\mathbb{L}:Dom(\mathbb{L})\to[W^{1-1/p}_{p}(\partial\Omega)]^{k} $ is continuous, clearly $\mathbb{L}$ is onto.  
 \end{proof}
 \item[III] $\mathbb{L}$ It is  a Fredholm operator of index zero 
 \begin{proof} Since
 \item[1]  $\ker{\mathbb{L}}=0$ has a finite dimension; $dim\ker{\mathbb{L}}=1$
 
 \item[2] $Im {\mathbb{L}} $ closed and has a finite co-dimension (i.e.; $R({\mathbb{L}})=[W^{1-1/p}_{p}(\partial\Omega)]^{k}$ ) 
 $dim ([W^{1-1/p}_{p}(\partial\Omega)]^{k}/R({\mathbb{L}}))=1$
  Indeed it is Fredholm operator of index zero  by the definition of Fredholm operator  of index zero.
   \end{proof}
   \item[IV] The inverse of $\mathbb{K}=\mathbb{L}^{-1}$ exists
   \begin{proof}
    From \item[II] the inverse $\mathbb{K}$ exists
   \end{proof}
\item[V] $\mathbb{K}$ is a compact
\begin{proof}
 Owing to the Compactness of the trace operator $Dom({\mathbb{L}})\hookrightarrow [W^{1-1/p}_{p}(\partial\Omega)]^{k},$ we
deduce that 
$$\mathbb{K}=\mathbb{L}^{-1}:[W^{1-1/p}_{p}(\partial\Omega)]^{k}\to^{\rm~ continuous~} Dom({\mathbb{L}})\hookrightarrow^{\rm compact}[W^{1-1/p}_{p}(\partial\Omega)]^{k}$$
 Since we that left compositions  and $T\circ S$ of compact $T$ with continuous $S$ produce compact operators for the proof see \cite{GT83}.
 Therefor $\mathbb{K}$ is a compact linear operator from $[W^{1-1/p}_{p}(\partial\Omega)]^{k}$ into $[W^{1-1/p}_{p}(\partial\Omega)]^{k}.$
\end{proof}

 \item[VI] $\v{N}$ is continuous operator 
 \begin{proof}
  Since the function $G$ satisfied $(G3)$ and $[W^{1-1/p}_{p}(\partial\Omega)]^{k}\subset [C(\partial\Omega)]^{k}$
  Through the surjectivity  of the trace operator $[W^{1-1/p}_{p}(\partial\Omega)]^{k}\to [W^{1-1/p}_{p}(\partial\Omega)]^{k}$ and the embedding 
  $[W^{1-1/p}_{p}(\partial\Omega)]^{k}\Subset [C(\partial\Omega)]^{k},$ for $p>N$ it follows that nonlinear operator $\v{N}$ is continuous operator.
 \end{proof}
\item[VII] The operator $\mathbb{K}\v{N}$ is continuous and compact (i.e.; completely continuous)
\begin{proof}
 $\mathbb{K}\v{N}:[W^{1-1/p}_{p}(\partial\Omega)]^{k}\to[W^{1-1/p}_{p}(\partial\Omega)]^{k}$ Since $\mathbb{K}$ is compact from $V:$,
 and $\v{N}$ is continuous from  $VI:$, then the compositions of $\mathbb{K}\v{N}$ is completely continuous (see \cite{Eva99}).
\end{proof}
   \end{description}
 Now we will discuss system (\ref{nlh}). Thus system \ref{nlh}  is equivalent to 
 \begin{equation}\label{nlh0}
\begin{gathered}
U=\lambda\mathbb{L}^{-1}\v{N}U \,{\rm with}\,\lambda\in[0,1]\,\,\, U\in [W^{1-1/p}_{p}(\partial\Omega)]^{k} \,{\rm or}\\
U=\lambda\mathbb{K}\v{N}U \,{\rm with}\,\lambda\in[0,1]\,\,\, U\in [W^{1-1/p}_{p}(\partial\Omega)]^{k}   \, {\rm or}\\
(I-\lambda\mathbb{K}\v{N})U=0 \,{\rm with}\,\lambda\in[0,1]\,\,\, U\in [W^{1-1/p}_{p}(\partial\Omega)]^{k} 
\end{gathered}
\end{equation}
Which shows that,  for  each $\lambda\in[0,1]$, the operator $\lambda\mathbb{K}\v{N}$ is a nonlinear compact perturbations of the 
identity on $[W^{1-1/p}_{p}(\partial\Omega)]^{k}$. It suffices to show that $\mathbb{K}\v{N}$ has fixed point $U\in[W^{1-1/p}_{p}(\partial\Omega)]^{k}. $
By  properties of $\mathbb{K}$, It follows that a fixed point $U\in Dom(\mathbb{L})$. Hence, $U\in[W^{2}_{p}(\Omega)]^{k},$
and is a (strong) solution of the nonlinear system (\ref{Es01}). We show that all possible solutions to the homotopy (\ref{hs01}) (equivalently to (\ref{hs02}) and (\ref{nlh})) are uniformly bounded 
in $[W^{1-1/p}_{p}(\partial\Omega)]^{k}$ independent of $\lambda\in[0,1]$ ( actually we show that they are bounded in $[W^{2}_{p}(\Omega)]^{k},$)
and then use the topological degree theory to show the existence of strong solution. We first prove following lemma 
which provides intermediate a priori estimates. 
\section{The A Priori Estimates }
\begin{lemma}\label{Lem0}
 Assume that $(A1),~(A2),~(S1),~(S2)~, (G1)-(G3)$ and assumption $\ref{ASMP}$ are holds,  assume there exist constants $a,b>0$,  
and Let $\alpha(x):=a(\min\lambda_{P}(x))$ and ,$\beta(x)=b(\max\lambda_{P}(x))$  such that 
$$\mu_{j}\leq\alpha(x)\leq\liminf_{|U|\to\infty}\frac{\<g(x,U),U\>}{|U|^{2}}\leq\limsup_{|U|\to\infty}\frac{\<g(x,U),U\>}{|U|^{2}}\leq \beta(x)\leq \mu_{j+1}$$
 uniformly for a.e. $x\in \partial\Omega$ with 
 $$\int_{\partial\Omega}(\alpha(x)-\mu_{j})\<\varphi(x),\varphi(x)\>\,dx>0,~~~\forall ~~\varphi\in E_{j+1}\setminus\{0\}$$
and $$~~\int_{\partial\Omega}(\mu_{j+1}-\beta(x))\<\varphi(x),\varphi(x)\>\,dx>0,~~\forall~~ \varphi\in E_{j+1}\setminus\{0\}, 
  $$ 
  Where $i\in \mathbb{N}, $ $E_{i}$ denotes the Steklov-Robin nullspace (finite-dimensional) with the Steklov-Robin
  eigenvalue $\mu_{i}$. Then all possible solutions to 
  \begin{equation}\label{hs0l}
\begin{gathered}
 -\Delta U + A(x)U = 0\quad\text{in } \Omega,\\
\frac{\partial U}{\partial \nu}+\Sigma(x)U=(1- \lambda) \delta P(x)U+\lambda g(x,U) \quad\text{on }
\partial\Omega,
\end{gathered}
\end{equation}
are (uniformly) bounded in $H(\Omega)$ independently of $\lambda\in[0,1]$
\end{lemma}
\begin{proof}
 By the way of contradiction. Assume the lemma doesn't hold. Then there are sequences $\{\lambda_{n}\}_{n=1}^{\infty}\subset[0,1]$ and  
$\{U_{n}\}_{n=1}^{\infty}\subset H(\Omega)$, such that $||U_{n}||_{(A,\Sigma)}\to\infty$ and put $U=U_{n}\,,$ $ \lambda=\lambda_{n}$ in \ref{hs01}  we have that 
\begin{equation}\label{hs02}
\begin{gathered}
 -\Delta U_{n} + A(x)U_{n} = 0\quad\text{in } \Omega,\\
\frac{\partial U_{n}}{\partial \nu}+\Sigma(x)U_{n}=(1- \lambda_{n}) \delta P(x)U_{n}+\lambda_{n} g(x,U_{n}) \quad\text{on }
\partial\Omega,
\end{gathered}
\end{equation}
Now multiple the first equation in \ref{hs02} by $V$  and integration by parts we have that 
\begin{equation}\label{hs03}
\begin{gathered}
 \int_{\Omega}\triangledown U_{n}.\triangledown V\,dx+\int_{\Omega}\<A(x)U_{n},V\>\,dx+\int_{\partial\Omega}\<\Sigma(x)U_{n},V\>\,dx= \\
(1-\lambda_{n})\delta\int_{\partial\Omega}\< P(x)U_{n},V\>\,dx+\lambda_{n}\int_{\partial\Omega}\<g(x,U_{n}),V\>\,dx ~~\forall~~V\in H(\Omega)
\end{gathered}
\end{equation}
 Let $V_{n}=\frac{U_{u}}{||U_{n}||_{(A,\Sigma)}}$ Clearly $V_{n}$ is bounded in $H(\Omega).$
 There exist a subsequence (relabeled) $V_{n}$  which convergent weakly to $V_{0}$ in $H(\Omega),$ and $V_{n}$ convergent strongly to $V_{0}$ 
 in $L^{2}(\partial\Omega)$. Without loss of generality $\lambda_{n}\to\lambda_{0}\in[0,1]$ Due to the at most linear growth condition on the boundary 
 nonlinearity $G$ (See conditions $(G1-G3$), it follows that 
 $\frac{\<g(x,U_{n}),V\>}{||U_{n}||_{(A,\Sigma)}}$ is bounded in $L^{2}(\partial\Omega)$
 Using the fact that $L^{2}(\partial\Omega)$  is a reflexive Banach space, by theorem \ref{rex} we get that  $\frac{\<g(x,U_{n}),V\>}{||U_{n}||_{(A,\Sigma)}}$  converges 
 weakly to $g_{0}$ in $L^{2}(\partial\Omega)$. Dividing equation (\ref{hs03}) by $||U_{n}||_{(A,\Sigma)}$
 
\begin{equation}\label{hs04}
\begin{gathered}
 \int_{\Omega}\frac{\triangledown U_{n}.\triangledown V}{||U_{n}||_{(A,\Sigma)}}\,dx+\int_{\Omega}\frac{\<A(x)U_{n},V\>}{||U_{n}||_{(A,\Sigma)}}
 \,dx+\int_{\partial\Omega}\frac{\<\Sigma(x)U_{n},V\>}{||U_{n}||_{(A,\Sigma)}}\,dx= \\
(1-\lambda_{n})\delta\int_{\partial\Omega}\frac{\< P(x)U_{n},V\>}{||U_{n}||_{(A,\Sigma)}}\,dx
+\lambda_{n}\int_{\partial\Omega}\frac{\<g(x,U_{n}),V\>}{||U_{n}||_{(A,\Sigma)}}\,dx ~~\forall~~V\in H(\Omega)
\end{gathered}
\end{equation}
we get that 
 \begin{equation}\label{hs05}
\begin{gathered}
 \int_{\Omega}\triangledown V_{n}.\triangledown V\,dx+\int_{\Omega}\<A(x)V_{n},V\>\,dx+\int_{\partial\Omega}\<\Sigma(x)V_{n},V\>\,dx= \\
(1-\lambda_{n})\delta \int_{\partial\Omega}\<P(x)V_{n},V\>\,dx+\lambda_{n}\int_{\partial\Omega}\frac{\<g(x,U_{n}),V\>}{||U_{n}||_{(A,\Sigma)}}\,dx ~~\forall~~V\in H(\Omega)
\end{gathered}
\end{equation}
 take the limit as $n\to\infty$ we get that 
 \begin{equation}\label{hs06}
\begin{gathered}
 \int_{\Omega}\triangledown V_{0}.\triangledown V\,dx+\int_{\Omega}\<A(x)V_{0},V\>\,dx+\int_{\partial\Omega}\<\Sigma(x)V_{0},V\>\,dx= \\
(1-\lambda_{0})\delta \int_{\partial\Omega}\<P(x)V_{0},V\>\,dx+\lambda_{0}\int_{\partial\Omega}\<g_{0},V\>\,dx ~~\forall~~V\in H(\Omega)
\end{gathered}
\end{equation}
 Let $V=V_{0}$ we  have that 
 \begin{equation}\label{hs07}
\begin{gathered}
 ||V_{0}||_{(A,\Sigma)}^{2}=\int_{\Omega}\triangledown V_{0}.\triangledown V_{0}\,dx+\int_{\Omega}\<A(x)V_{0},V_{0}\>\,dx+\int_{\partial\Omega}\<\Sigma(x)V_{0},V_{0}\>\,dx= \\
(1-\lambda_{0})\delta\int_{\partial\Omega}\< P(x)V_{0},V_{0}\>\,dx+\lambda_{0}\int_{\partial\Omega}\<g_{0},V_{0}\>\,dx ~~\forall~~V\in H(\Omega)
\end{gathered}
\end{equation}
 Therefore 
 \begin{equation}\label{hs08}
\begin{gathered}
 ||V_{0}||_{(A,\Sigma)}^{2}=(1-\lambda_{0})\delta\int_{\partial\Omega}\< P(x)V_{0},V_{0}\>\,dx+\lambda_{0}\int_{\partial\Omega}\<g_{0},V_{0}\>\,dx ~~\forall~~V\in H(\Omega)
\end{gathered}
\end{equation}
Now let $V=V_{n}$ in equation (\ref{hs05}) we get that 
$$1=||V_{n}||_{(A,\Sigma)}=
(1-\lambda_{n})\delta \int_{\partial\Omega}\<P(x)V_{n},V_{n}\>\,dx+\lambda_{n}\int_{\partial\Omega}\frac{\<g(x,U_{n}),V_{n}\>}{||U_{n}||_{(A,\Sigma)}}\,dx
$$
Taking the limit as $n\to\infty$ and using equation (\ref{hs08}) and the fact that $\frac{\<g(x,U_{n}),V_{n}\>}{||U_{n}||_{(A,\Sigma)}}$ converges
weakly to $g_0$ in $L^{2}(\partial\Omega)$ 
we have that 
\begin{equation}\label{hs09}
\begin{gathered}
 ||V_{0}||_{(A,\Sigma)}^{2}=(1-\lambda_{0})\delta\int_{\partial\Omega}\< P(x)V_{0},V_{0}\>\,dx+
 \lambda_{0}\int_{\partial\Omega}\<g_{0},V_{0}\>\,dx=1 ~~\forall~~V\in H(\Omega)
\end{gathered}
\end{equation}
\begin{claim}
 $V_{0}=0$ which will lead to a contradiction 
 \end{claim}
 \begin{proof}
 From equation (\ref{hs09}), notice that $V_{0}$ is a weak solution of the linear equation 
 \begin{equation}\label{hs010}
\begin{gathered}
 -\Delta U+A(x)U=0~~a.e.~~{\rm in}~\Omega\\
 \frac{\partial U}{\partial\nu}+\Sigma(x) U=(1-\lambda_{0})\delta P(x)U+
 \lambda_{0}g_{0}~~~~{\rm on}~\partial\Omega
\end{gathered}
\end{equation}
Let us mention here that equation (\ref{hs010}) implies that $\lambda_{0}\neq0.$ Otherwise, Since $\delta$ is in the Steklov-Robin resolvent, we deduce
that $V_{0}=0;$ which contradicts the fact that  $||V_{0}||_{(A,\Sigma)}^{2}=1$.
In order to bring out all the properties of the function $V_{0}$, we need to analyze a little bit more carefully the function 
$(1-\lambda_{0})\delta P(x)V_{0}(x)+\lambda_{0}g_{0}(x).$ Let us denote by $k(x)$ the function defined by 
\begin{equation}\label{hs011}
\begin{gathered}
   k(x) = \left\{
     \begin{array}{lr}
        (1-\lambda_{0})\delta P(x)V_{0}+\lambda_{0}\frac{g_{0}(x)}{||V_{0}||_{(A,\Sigma)}}& {\rm if}~~  V_{0}(x)\neq0\\
       0 & {\rm if}~~  V_{0}(x)=0
     \end{array}
   \right.
\end{gathered}
\end{equation}
From the definition of $\delta$ and the condition in theorem \ref{th0}, it turns out that 
\begin{equation}\label{hs012}
\begin{gathered}
 \mu_{j}\leq \alpha(x)\leq\<k(x)V_{0},V_{0}(x)\>\leq\beta(x)\leq\mu_{j+1} ~~{\rm for}~~V_{0}(x)\neq0
\end{gathered}
\end{equation}
Therefore $V_0$ is a weak solution to the linear equation 
\begin{equation}\label{hs013}
\begin{gathered}
    \left\{
     \begin{array}{lr}
        -\Delta U+A(x) U=0& a.e.~~{\rm in}~~\Omega\\
       \frac{\partial U}{\partial\nu}+\Sigma(x) U=k(x)U& {\rm on}~~\partial\Omega
     \end{array}
   \right.
\end{gathered}
\end{equation}
that is 
\begin{equation}\label{hs015}
\begin{gathered}
    \int_{\Omega}\triangledown V_{0}.\triangledown V\,dx+\int_{\Omega}\<A(x)V_{0},V\>\,dx+\int_{\partial\Omega}\<\Sigma(x)V_{0},V\>\,dx=
    \int_{\partial\Omega}\<k(x)V_{0},V\>\,dx ~~ \forall~~ V\in H(\Omega)
\end{gathered}
\end{equation}
we claim that this implies that either $V_{0}\in E_{j}$ or $V_{0}\in E_{j+1}$ only ( see Lemma \ref{lem2} below ). Let us assume that for 
the  time  being that this holds and finish the proof.\\ 
If $V_{0}\in E_{j},$ then taking $V=V_{0}$ in equation (\ref{hs015}) we have that 
$$\mu_{j}\int_{\partial\Omega}\<V_{0},V_{0}\>\,dx=||V_{0}||^{2}_{(A,\Sigma)}=\int_{\partial\Omega}\<k(x)V_{0},V_{0}\>\,dx.$$ Using equation (\ref{hs012}) 
we get that $\int_{\partial\Omega}(\alpha(x)-\mu_{j})\<V_{0},V_{0}\>\,dx\leq0$. Since 
$\int_{\partial\Omega}(\alpha(x)-\mu_{j})\<\varphi,\varphi\>>0\,dx\leq0$. for all $\varphi\in E_{j}\setminus\{0\},$ we conclude that 
$V_{0}=0$ which contradicts the fact $||V_{0}||^{2}_{(A,\Sigma)}=1.$\\
Similarly, if $V_{0}\in E_{j+1}$, then taking $V=V_{0}$ in equation (\ref{hs015}) we get that 
 $$\int_{\partial\Omega}(\mu_{j+1}-\beta(x))\<V_0,V_0\>\leq0$$. Since 
 
 $\int_{\partial\Omega}(\mu_{j+1}-\beta(x))\<\varphi,\varphi\>>0\,dx\leq0$. for all $\varphi\in E_{j+1}\setminus\{0\},$ we conclude that 
$V_{0}=0$ which contradicts the fact $||V_{0}||^{2}_{(A,\Sigma)}=1$ again. \\
The proof of the claim is complete.
 \end{proof}
 Thus all possible solutions of the homotopy (\ref{hs01}) are uniformly bounded in $H(\Omega)$ independently of $\lambda \in [0,1]$. 
 The proof is complete.
 \end{proof}
 The following lemma provide some useful information about that function $V_{0}$ that was used in the proof of preceding lemma 
 \begin{lemma}\label{lem2}
  If $U$ is a ( nontrivial) weak solution of equation (\ref{hs013}) with 
  $\mu_{j}\leq \alpha(x)\leq\<k(x)V(x),V(x)\>\leq\beta(x)\leq\mu_{j+1}$, for all $V\in H(\Omega)$  then either $U\in E_{j}$ or $U\in E_{j+1}.$
 \end{lemma}
 \begin{proof}
  Since $U$ is (also) a weak solution, it satisfies 
  \begin{equation}\label{hs0151}
\begin{gathered}
    \int_{\Omega}\triangledown U.\triangledown V\,dx+\int_{\Omega}\<A(x)U,V\>\,dx+\int_{\partial\Omega}\<\Sigma(x)U,V\>\,dx=
    \int_{\partial\Omega}\<k(x)U,V\>\,dx ~~ \forall~~ V\in H(\Omega)~~~\forall~~V\in H(\Omega)
\end{gathered}
\end{equation}
Observe that $U\in [H_{(M.P)}(\Omega)]^{\bot}.$
Hence $U=U_{1}+U_{2}$ where $U_{1}\in \oplus_{l\leq j}E_{l}$ and $U_{2}\in \oplus_{l\geq j+1}E_{l}$ know that from the properties of the Steklov-Robin
eigenfunction
\begin{equation}\label{hs016}
\begin{gathered}
  ||U_{1}||^{2}_{(A,\Sigma)}\leq \mu_{j}\int_{\partial\Omega}\<P(x)U_{1},U_{1}\>,\,\,{\rm ~for~ all}\,\, U_{1}\in\oplus_{l\leq j}E_{l}\,\,~~ {\rm and}\\ 
||U_{2}||^{2}_{(A,\Sigma)}\geq \mu_{j+1}\int_{\partial\Omega}\<P(x)U_{2},U_{2}\>,\,\,{\rm ~for~ all}\,\, U_{2}\oplus_{l\geq j+1}E_{l}
\end{gathered}
\end{equation}
Tanking $V=U_{1}-U{2}$, in equation (\ref{hs0151} we get that 
\begin{equation}\label{hs017}
\begin{gathered}
    \int_{\Omega}|\triangledown U_{1}|^{2}\,dx+\int_{\Omega}\<A(x)U_{1},U_{1}\>\,dx+\int_{\partial\Omega}\<\Sigma(x)U_{1},U_{1}\>\,dx
    \\-\left(
    \int_{\Omega}|\triangledown U_{2}|^{2}\,dx+\int_{\Omega}\<A(x)U_{2},U_{2}\>\,dx+\int_{\partial\Omega}\<\Sigma(x)U_{2},U_{2}\>\,dx
    \right)
        =\\
    \int_{\partial\Omega}\<k(x)U_{1},U_{1}\>\,dx- \int_{\partial\Omega}\<k(x)U_{2},U_{2}\>\,dx ~~ \forall~~ V\in H(\Omega)~~~\forall~~V\in H(\Omega)
\end{gathered}
\end{equation}
Using equation (\ref{hs0151}), we obtain that 
$$\int_{\partial\Omega}(\<k(x)U_{1},U_{1}\>-\mu_{j}\<P(x)U_{1},U_{1}\>)\<U_{1},U_{1}\>\,dx+
\int_{\partial\Omega}(\<k(x)U_{2},U_{2}\>-\mu_{j+1}\<P(x)U_{2},U_{2}\>)\<U_{2},U_{2}\>\,dx\leq0$$ Therefore 
$$\int_{\partial\Omega}(\<k(x)U_{1},U_{1}\>-\mu_{j}\<P(x)U_{1},U_{1}\>)\<U_{1},U_{1}\>\,dx=0$$ and 
$$\int_{\partial\Omega}(\<k(x)U_{2},U_{2}\>-\mu_{j+1}\<P(x)U_{2},U_{2}\>)\<U_{2},U_{2}\>\,dx=0$$
Let 
$$S_{1}:=\{x\in\partial\Omega:U_{1}(x)\neq0\}$$ and 
$$S_{2}:=\{x\in\partial\Omega:U_{2}(x)\neq0\}.$$ It follows that 
$\<k(x)U_{1},U_{1}\>=\mu_{j}\<P(x)U_{1},U_{1}\>~a.e.~~{\rm on}~~~S_{1}$ and 
$\<k(x)U_{2},U_{2}\>=\mu_{j+1}\<P(x)U_{2},U_{2}\>~a.e.~~{\rm on}~~~S_{2}.$
If $meas(S_{1}\cap S_{2})>0$ we have that 
$\<k(x)U_{1},U_{1}\>=\mu_{j}\<P(x)U_{1},U_{1}\>= 
\<k(x)U_{2},U_{2}\>=\mu_{j+1}\<P(x)U_{2},U_{2}\>$ for a.e.;  $x\in S_{1}\cap S_{2}$ which cannot happen since 
 $\mu_{j}\neq\mu_{j+1}.$
 Now assume that $meas(S_{1}\cap S_{2})=0$ that is $U_{2}(x)=0~~a.e.; {\rm on}~~~S_{1}$ and 
 $U_{1}(x)=0~~a.e.; {\rm on}~~~S_{2}.$
 If $U_{1}\neq0$, then taking $V=U_{1}$ in equation (\ref{hs0151}) and using $(A,\Sigma)-$orthogonality, we get that 
 $$\int_{\Omega}|\triangledown U_{1}|^{2}\,dx+\int_{\Omega}\<A(x)U_{1},U_{1}\>\,dx+\int_{\partial\Omega}\<\Sigma(x)U_{1},U_{1}\>\,dx=
  \int_{\partial\Omega}\<k(x)U_{1},U_{1}\>\,dx.$$
 Since $\mu_{j}\<P(x)U_{1},U_{1}\>\leq\<k(x)U_{1},U_{1}\>,$ we have that 
$||U_1||^{2}_{(A,\Sigma)}\geq\mu_{j}\int_{\partial\Omega}\<P(x)U_{1},U_{1}\>\,dx.$ It follows from equation (\ref{hs016})
$||U_1||^{2}_{(A,\Sigma)}=\mu_{j}\int_{\partial\Omega}\<P(x)U_{1},U_{1}\>\,dx.$ which implies that $U_1\in E_{j}.$\\
Similarly, If $U_{2}\neq0$ then taking $V=U_{2}$ in equations (\ref{hs0151}) and using the $(A,\Sigma)-$orthogonality, we get that 
 $$\int_{\Omega}|\triangledown U_{2}|^{2}\,dx+\int_{\Omega}\<A(x)U_{2},U_{2}\>\,dx+\int_{\partial\Omega}\<\Sigma(x)U_{2},U_{2}\>\,dx=
  \int_{\partial\Omega}\<k(x)U_{2},U_{2}\>\,dx.$$
   Since $\mu_{j+1}\<P(x)U_{1},U_{1}\>\geq\<k(x)U_{1},U_{1}\>,$ we have that 
  $||U_2||^{2}_{(A,\Sigma)}\leq\mu_{j+1}\int_{\partial\Omega}\<P(x)U_{2},U_{2}\>\,dx.$  It follows from equation (\ref{hs016}) that 
  $$||U_2||^{2}_{(A,\Sigma)}=\mu_{j+1}\int_{\partial\Omega}\<P(x)U_{2},U_{2}\>\,dx$$
  which implies that $U_2\in E_{j+1}.$\\
  Thus $U=U_{1}+U_{2}$ with $U_1\in E_{j},$ and $U_2\in E_{j+1}.$\\
  Finally, we claim that the function $U$ cannot be written in the $U=U_{1}-U_{2}$ where $U_{1}\in E_{j}\setminus\{0\}$ and 
  $U_{2}\in E_{j+1}\setminus\{0\}.$ Indeed, suppose that this does not hold; that is, $U=U_{1}+U_{2}$ with 
  $U_{1}\in E_{j}\setminus\{0\}$ and 
  $U_{2}\in E_{j+1}\setminus\{0\}.$ Then, by taking $V=U_1-U_2$ in equation (\ref{hs0151}), we get equation (\ref{hs017})
  Since $U_{1}\in E_{j}$ and $U_{2}\in E_{j+1}$ and $$\alpha(x)\leq\<k(x)U,U\>\leq\beta(x)~a.e.; ~{\rm on}~~\partial\Omega,$$ 
  we deduce 
  $$\int_{\partial\Omega}(\alpha(x)-\mu_{j})\<U_{1},U_{1}\>\,dx\leq 0~~{\rm and}~~
  \int_{\partial\Omega}(\mu_{j+1}-\beta(x))\<U_{2},U_{2}\>\,dx\leq 0$$
  Which contradicts the fact that 
  $\int_{\partial\Omega}(\alpha(x)-\mu_{j})\<\varphi,\varphi\>\,dx>0$ for all $\varphi\in E_{j}\setminus\{0\}$
  and  $\int_{\partial\Omega}(\mu_{j+1}-\beta(x))\<\varphi,\varphi\>\,dx> 0$ for all $\varphi\in E_{j+1}\setminus\{0\}$
  Thus, either $U\in E_{j}$ or $U\in E_{j+1}$. The proof is complete
    \end{proof}
    \section{Proof of the main Theorem \ref{th0}}
Let $(\lambda, U)\in [0,1]\times [W^{1-1/p}_{p}(\partial\Omega)]^{k}$ be a solution to the homotopy (\ref{nlh0})
(equivalently (\ref{hs01})). Since 
$\int_{\Omega}\triangledown U.\triangledown V\,dx+\int_{\Omega}\<A(x)U,V\>\,dx+\int_{\partial\Omega}\<\Sigma(x)U,V\>\,dx=
0$ for all $U\in [C_{0}^{1}(\Omega)]^{k},$ and the trace of $U\in [W^{1-1/p}_{p}(\partial\Omega)]^{k}\subset [C_{0}^{1}(\partial\Omega)]^{k}$ 
It follows from Theorem 13.1 in [\cite{LU1968}, pp.1999-200] also see [\cite{Eva99}, \cite{Bz}]

that there is a constant $C_{0}>0$
(independent of $U$) such that 
$\sup_{\Omega}|U(x)|\leq C_{0}||U||_{H(\Omega)},$ and $\max_{\bar{\Omega}}|U(x)|\leq C_{0}||U||_{H(\Omega)},$
by the continuity of $U$ on $\bar{\Omega}$  from Lemma \ref{Lem0} above and the ( local Lipschitz) continuity of $g$ 
we deduce that 
$\max_{\partial\Omega}|\frac{\partial U}{\partial\nu}+\Sigma U|=\max_{\partial\Omega}|(1-\lambda)\delta P(.)U+\lambda g(.,U)$
is bounded independently of $U$ and $\lambda$. Actually, we deduce from Theorem 2 in (\cite{Lie1988} p 1204) that 
$|U|_{[C^{1}(\bar{\Omega})]^{k}}$ is bounded (independently of $U$ and $\lambda$). Therefore, the continuity of the trace operator 
$[C^{1}(\bar{\Omega})]^{k}\subset[W^{1}_{p}(\Omega)]^{k}\to [W^{1-1/p}_{p}(\partial\Omega)]^{k}$
and Lemma \ref{Lem0} herein imply that there is a constant $C_{1}>0$ ( independent of $U$ and $\lambda$) such that 

\begin{equation}\label{hs018}
\begin{gathered}
  ||U||_{[W^{1-1/p}_{p}(\partial\Omega)]^{k}}< C_{1}
\end{gathered}
\end{equation}
for all possible solutions to the homotopy \ref{hs01}. Now by the homotopy invariance property of the topological degree 
(see \cite{L}, \cite{Ma}) it follows that 
$$1=deg(I,B_{C_1}(0),0)=deg(I-\mathbb{K}\v{N},B_{C_1}(0),0)\neq0$$
where $B_{C_1}\subset[W^{1-1/p}_{p}(\partial\Omega)]^{k} $  is ball of radius $C_{1}>0$ centered at the origin. Thus, by the 
existence property of the topological degree (see \cite{L}, \cite{Ma}), the nonlinear operator $\mathbb{K}\v{N}$
has a fixed point in $[W^{1-1/p}_{p}(\partial\Omega)]^{k}$  which is also in $[W^{2}_{p}(\Omega)]^{k}$ as aforementioned. Thus 
Then the nonlinear system (\ref{Es02}) has at least one (strong) solution $U\in [W^{2}_{p}(\Omega)]^{k}$. The proof of the theorem is complete.
 \subsection{Remarks}
 \begin{remark}
  Notice that, since $g$ is locally Lipschitz, it follows from \ref{hs018} and the boundary condition in the homotopy \ref{hs01}
  that $||U||_{[W^{2-1/p}_{p}(\partial\Omega)]^{k}}< C_{2}$ for some constant $C_{2}>0$ independent of $U$ and $\lambda$. Therefore
  $||U||_{[W^{2}_{p}(\partial\Omega)]^{k}}< C_{3}$ for some constant $C_{3}>0$
 \end{remark}
 \begin{remark}
  The case $\mu_{j}=\mu_{1}$ more clearly illustrates the fact that the non-resonance conditions in Theorem \ref{th0} are genuinely 
  of nonuniform type. Indeed, in this case $E_{1}\backslash\{0\}$ contains only (continuous) functions which are either positive 
  or negative on $\bar{\Omega}.$ The condition that $\alpha(x)\geq\mu_{1}~~a.e.~~{\rm on}~~\partial\Omega$ with 
  $\int_{\partial\Omega}(\alpha(x)-\mu_{1})\<\varphi,\varphi\>$ for all $\varphi\in E_{1}\backslash\{0\}$ is equivalent to saying 
  that $\alpha(x)\geq\mu_{1}~a.e.~~{\rm on}~~\partial\Omega$ with strict inequality on a subset of $\partial\Omega$ of positive measure. 
  Thus $\alpha(x)$ need not be (uniformly) bounded away from $\mu_{1}.$
 \end{remark}
 \begin{remark}
  Our main result, Theorem \ref{th0} herein, Still holds true when $A(x)\equiv0$ and $\Sigma(x)\equiv0$ ( The Laplace's system in the original linear system 
  with was considered by Auchmuty for the system \cite{Auc12}, \cite{MK2},  and in scalar case was considered by Steklov on a disk in \cite{Sk}.)
  Indeed, a modification is needed in the proof of Lemma \ref{Lem0} as follows. We proceed as in the proof with $||.||_{(A,\Sigma)}$ ( here 
 $||.||_{H(\Omega)}$ denotes the  standard $H(\Omega)-$norm, and $V_{n}=\frac{U_{n}}{||U||_{H(\Omega)}}$ up to the equation (\ref{hs06})
 Taking $V=V_{0}$ in \ref{hs06} we now get 
 \begin{equation}\label{hs060}
\begin{gathered}
 \int_{\Omega}|\triangledown V_{0}|^{2}\,dx= 
(1-\lambda_{0})\delta \int_{\partial\Omega}\<P(x)V_{0},V_{0}\>\,dx+\lambda_{0}\int_{\partial\Omega}\<g_{0},V_{0}\>\,dx 
\end{gathered}
\end{equation}
Now, taking $V=\frac{U_{n}}{||U_{n}||_{H(\Omega)}}$ in \ref{hs07} where is replaced by $||U_{n}||_{H(\Omega)}$
we get that 
\begin{equation}\label{hs070}
\begin{gathered}
 \int_{\Omega}|\triangledown V_{n}|^{2}\,dx= 
(1-\lambda_{n})\delta \int_{\partial\Omega}\<P(x)V_{n},V_{n}\>\,dx+\lambda_{n}\int_{\partial\Omega}\frac{\<g(x,U_{n}),V_{n}\>}{||U_{n}||_{H(\Omega)}}\,dx 
\end{gathered}
\end{equation}
Tanking the limit as $n\to\infty$ and using \ref{hs060} and the fact that $\frac{\<g(x,U_{n}),V_{n}\>}{||U_{n}||_{H(\Omega)}}$ converges weakly 
$g_0$ in $[L^{2}(\partial\Omega)]^{k}$ $V_{n}$ converges strongly to $V_{0}$ in $[L^{2}(\partial\Omega)]^{k}$, we have that 
\begin{equation}\label{hs071}
\begin{gathered}
\lim_{n\to\infty} \int_{\Omega}|\triangledown V_{n}|^{2}\,dx= 
 (1-\lambda_{0})\delta \int_{\partial\Omega}\<P(x)V_{0},V_{0}\>\,dx+\lambda_{0}\int_{\partial\Omega}\<g_{0},V_{0}\>\,dx=
 \int_{\Omega}|\triangledown V_{0}|^{2}\,dx
\end{gathered}
\end{equation}
This implies that 
\begin{equation}\label{hs072}
\begin{gathered}
 ||V_{0}||_{H(\Omega)}^{2}=\int_{\Omega}|\triangledown V_{0}|^{2}\,dx+\int_{\Omega}|V_{0}|^{2}\,dx=
 \lim_{n\to\infty}\left(\int_{\Omega}|\triangledown V_{n}|^{2}\,dx+\int_{\Omega}|V_{n}|^{2}\,dx\right)=
 \lim_{n\to\infty}||V_{n}||_{H(\Omega)}^{2}=1
\end{gathered}
\end{equation}
We now proceed as in the proof of Lemma \ref{Lem0} after equation (\ref{hs09}) to Show that $V_{0}=0$; which is a contradiction with equation 
\ref{hs072}. \\
The proof of Lemma \ref{lem2} also needs to be modified as follows. The norm $||.||_{(A,\Sigma)}$ is now replaced by the $H(\Omega)-$ equivalent 
norm $||.||$ by 
$$||U||^{2}:=\int_{\Omega}|\triangledown U|^{2}\,dx+\int_{\Omega}|U|^{2}\,dx ~~~{\rm for }~~U\in H(\Omega)$$
(See \cite{Au4}) By the using the decomposition of $H(\Omega)$ given in [\cite{Au4} Theorem 7.3, p337 ] we now proceed with the argument used 
in the proof Lemma \ref{lem2} herein to reach its conclusion. 
 \end{remark}
 \begin{remark}
  Note that the case $A(x)\equiv0$ and $\Sigma(x)\equiv0$ even more clearly illustrates the fact that nonresonance conditions in Theorem \ref{th0}
  are genuinely of nonuniform type. Indeed, in this case $\mu_{1}=0$ and $E_{1}\backslash\{0\}$ contains only constant functions.
  The condition that $\alpha(x)\geq\mu_{1}~~{\rm on}~~\partial\Omega$ with strict inequality on a subset of $\partial\Omega$ of positive measure.
  Thus $\alpha(x)$ need not be (uniformly) bounded away from $\mu_{1}=0.$ Actually, a careful analysis of the proofs 
  
  of Lemmas \ref{Lem0} and \ref{lem2} shows that, in this case one can drop the requirement that $\alpha(x)\geq0~~~{\rm on}~~\partial\Omega$ 
  and require only that $\int_{\partial\Omega}\alpha(x)\,dx>0.$ Thus a 'crossing' of the zero eigenvalue on subset of $\partial\Omega$
  of positive measure is allowed that is $\alpha(x)$ could be negative on a subset of $\partial\Omega$ of positive measure (i.e $P(x)\leq0$)

 \end{remark}

\end{document}